\documentclass[12pt]{article}
\usepackage[T1]{fontenc}
\usepackage{lmodern}
\usepackage{amsmath}
\usepackage{amssymb}
\usepackage{latexsym}
\usepackage{amsthm}
\usepackage{mathrsfs}
\usepackage{enumitem}
\usepackage[colorlinks,
      linkcolor=blue,
      anchorcolor=blue,
      citecolor=blue
      ]{hyperref}
\usepackage{url}

\newtheorem{thm}{Theorem}
\newtheorem{lem}[thm]{Lemma}

\newtheorem{cor}[thm]{Corollary}
\newtheorem{conj}[thm]{Conjecture}

\newcommand{\ceil}[1]{\ensuremath{\left\lceil #1 \right\rceil}} 
\newcommand{\floor}[1]{\ensuremath{\left\lfloor #1 \right\rfloor}} 

\newcommand{\KL}{Kazhdan-Lusztig\ }
\DeclareMathOperator{\ch}{ch}

\newcommand{\grVRep}{\operatorname{grVRep}}

\newcommand{\VRep}{\operatorname{VRep}}

\DeclareMathOperator{\Ind}{Ind}

\begin{document}

\begin{center}
	{\large \bf Equivariant Kazhdan-Lusztig  polynomials of thagomizer matroids}
\end{center}


\begin{center}
Matthew H.Y. Xie\textsuperscript{a}, Philip B. Zhang\textsuperscript{b,}\footnote[1]{Corresponding author}\\[6pt]
\end{center}

\begin{center}
$^a$College of Science\\
Tianjin University of Technology, Tianjin 300384, P.R. China

$^{b}$College of Mathematical Science \\
Tianjin Normal University, Tianjin 300387, P. R. China\\[6pt]
\end{center}

\begin{center}
Email: $^{a}$\texttt{xie@mail.nankai.edu.cn}, $^{b}$\texttt{zhang@tjnu.edu.cn}
\end{center}

\noindent\textbf{Abstract.}
The equivariant Kazhdan-Lusztig polynomial of a matroid was introduced by Gedeon,  Proudfoot,  and Young.  Gedeon conjectured an explicit formula for the equivariant Kazhdan-Lusztig polynomials of thagomizer matroids with an action of symmetric groups.  In this paper, we discover a new formula for these polynomials which is related to the equivariant Kazhdan-Lusztig polynomials of uniform matroids.
Based on our new formula, we confirm Gedeon's conjecture by the Pieri rule.

\noindent {AMS Classification 2010:} 05B35, 05E05, 20C30.


\section{Introduction}

Given a matroid $M$, Elias, Proudfoot, and Wakefield~\cite{elias2016kazhdan} introduced  the Kazhdan-Lusztig polynomial $P_M(t)$. 
If $M$ is equipped with an action of a finite group $W$, 
Gedeon,  Proudfoot, and Young~\cite{gedeon2017equivariant} defined the $W$-equivariant
Kazhdan-Lusztig polynomial $P_M^W (t)$, whose coefficients are  graded virtual representations of $W$ and from which $P_M(t)$ can be recovered by sending virtual representations to their dimensions. 
The equivariant Kazhdan-Lusztig polynomials  have been computed for uniform matroids~\cite{gedeon2017equivariant}  and $q$-niform matroids~\cite{proudfoot2018equivariant}, and conjectured  for thagomizer  matroids~\cite{gedeon2016thagomizer}.

The thagomizer matroid $M_n$ is isomorphic to the  graphic matroid of  the   complete tripartite graph $K_{1,1,n}$ or  the graph obtained  by adding an edge between the two distinguished vertices of bipartite graph $K_{2,n}$.
Gedeon~\cite{gedeon2016thagomizer} computed the polynomial $P_{M_n}(t)$ and presented a conjecture for the  equivariant polynomial $P_{M_n}^{S_n}(t)$, where $S_n$ is the symmetric group of order $n$.
Let $\Upsilon_n$ be  the set of partitions of $n$ of the form $(a, n-a-2i-{\eta}, 2^i, \eta)$, where $\eta \in \{0,1\}$, $i\geq 0$ and $1<a < n$. For any partition $\lambda$ of $n$, we let $V_\lambda$ denote the irreducible representation of $S_n$ indexed by $\lambda$. 
We also set \[
\kappa(\lambda) : = \left\{
\begin{array}{l l}
\lambda_1-1, & \lambda = (n-1,1),\\
\lambda_1 - \lambda_2 +1, \quad &\mbox{otherwise},
\end{array}
\right.
\] 
and
\[
\omega(\lambda) : = \left\{
\begin{array}{l l}
0, & \lambda_{\ell(\lambda)}  = 1, \\
1, \quad & \mbox{otherwise}. 

\end{array}
\right.
\]
Gedeon \cite{gedeon2016thagomizer} conjectured an explicit formula for $P_{M_n}^{S_n}(t)$. 

\begin{conj}\label{conj-main} 
For any  positive integer $n$,
\begin{align*}
P_{M_n}^{S_n}(t)= \sum_{\lambda \in \Upsilon_n} \kappa(\lambda) V_\lambda t^{\ell(\lambda)-1} (t+1)^{\omega(\lambda)} + V_{(n)} ((n-1)t+1).
\end{align*}
\end{conj}

In this paper, we shall  confirm Conjecture~\ref{conj-main}. 
To this end, we  find  a new  formula for $P_{M_n}^{S_n}(t)$ which is related to the equivariant Kazhdan-Lusztig polynomials of uniform matroids.
 Let $U_{1,n}$ be the uniform matroid of rank $n$ on $n+1$ elements, which is  isomorphic  to the graphic matroid of the cycle graph with $n+1$ vertices. 
One of the main results of this paper is as follows. 
\begin{thm}\label{thm-rep}
For any  positive integer $n$, we have
	\begin{align}\label{eq-new}
	P_{M_n}^{S_n}(t)=V_{(n)}+t \sum_{k=2}^{n}    \Ind_{S_{n-k}\times S_{k} }^{S_{n}} \left( V_{(n-k)}\otimes P_{U_{1,k-1}}^{S_k}(t) \right),
	\end{align}
	where $P_{U_{1,k-1}}^{S_k}(t)=\sum_{i=0}^{\lfloor \frac{k}{2} \rfloor -1} V_{k-2i,2^i} t^{i}$.
\end{thm}

Note that for any parition $\lambda$ of $k$  there holds that
\begin{align}
\dim  \Ind_{S_{n-k}\times S_{k} }^{S_{n}} \left( V_{(n-k)}\otimes  V_{\lambda}\right)&=|S_{n}:S_{n-k}\times S_{k} |  \times \dim V_{(n-k)} \times \dim V_{\lambda}\notag\\
&=\frac{n!}{{(n-k)!}{k!}}\dim V_{\lambda}=\binom{n}{k}\dim V_{\lambda}\label{ind-rep},
\end{align}
where $|S_{n}:S_{n-k}\times S_{k}|$  is the index of $S_{n-k}\times S_{k}$ in $S_{n}$ in the sense of isomorphism.
Hence, the following formula for the non-equivalent \KL polynomials which inspires this paper, can be derived  from Theorem~\ref{thm-rep}.
\begin{cor}\label{relation-t-u}
For any  positive integer $n$, we have
\begin{align}\label{connection}
P_{M_n}(t) = 1 + t \sum_{k=2}^n \binom{n}{k} P_{U_{1,k-1}}(t).
\end{align}
\end{cor}

This paper is organized as follows.
Section~2 is dedicated to the proof of Theorem \ref{thm-rep}. The main tools used in our proof of Theorem~\ref{thm-rep} are the  Frobenius characteristic map and  the generating functions of  symmeric functions.
In Section~3, based on Theorem~\ref{thm-rep}, we confirm  Conjecture~\ref{conj-main} by the Pieri rule. 

%


\section{Proof of Theorem~\ref{thm-rep}}
In this section, we shall prove Theorem~\ref{thm-rep}. We first review the definition of Frobenius characteristic map and then show in Theorem~\ref{thm-pq} that Theorem~\ref{thm-rep}  can be translated into a symmetric function equality.
Once Theorem~\ref{thm-pq} is proved,  the proof of Theorem~\ref{thm-rep} is done since they are equivalent under the Frobenius characteristic map .

 Following Gedeon,  Proudfoot, and  Young  \cite{proudfoot2018equivariant},  let $\VRep(S_n)$ be the $\mathbb{Z}$-module  of isomorphism classes of virtual representations of $S_n$ and set
$ \text{grVRep}(W) := \text{VRep}(W)\otimes_{\mathbb{Z}}\mathbb{Z}[t].$
Consider the Frobenius characteristic map 
$$\ch : \grVRep(S_n) {\longrightarrow} \Lambda_n\otimes_{\mathbb{Z}}\mathbb{Z}[t],$$
where $\Lambda_n$ is the $\mathbb{Z}$-module of symmetric functions  of  degree $n$ in the variables $\mathbf{x}= \left( x_1, x_2, \ldots \right)$, see \cite[Section I.7]{Macdonald1995Symmetric}. 
We refer the reader to~\cite{Macdonald1995Symmetric, Stanley1999Enumerative} for undefined terminology from the theory of symmetric functions.
Given two  graded virtual representations $V_1\in \grVRep(S_{n_1})$ and $V_2\in \grVRep(S_{n_2})$, we have
$$\ch \left(\Ind_{S_{n_1}\times S_{n_2}}^{S_{n_1+n_2}}V_1\otimes V_2\right) = \ch(V_1)\ch(V_2).$$ 
The image of the irreducible representation  $V_{\lambda}$ under  $\ch$  is the Schur function $s_{\lambda}$ and, in particular, the image of the trivial representation $V_{(n)}$ is the  complete symmetric function $h_n(\mathbf{x})$. 
Define  $Q_n(\mathbf{x};t)$ as 
\begin{align}\label{def-Q}
Q_n(\mathbf{x};t) := 
\begin{cases}
0,& n=0 \ \mbox{or} \ 1,\\
\sum_{i=0}^{\lfloor \frac{n}{2} \rfloor -1} s_{n-2i,2^i}(\mathbf{x}) t^{i} , & n\ge 2.
\end{cases}
\end{align}
When $n\ge 2$, $Q_n(\mathbf{x};t)$ is  the image under the Frobenius map of  $P_{U_{1,n-1}}^{S_n}(t)$, see~\cite{proudfoot2016intersection}. 
Since the  Frobenius characteristic map   is an  isomorphism between $\grVRep(S_n)$  and $\Lambda_n\otimes_{\mathbb{Z}}\mathbb{Z}[t]$, the following theorem is equivalent to Theorem~\ref{thm-rep}.
\begin{thm}\label{thm-pq}
For any positive  integer $n$, we have
\begin{align}\label{eq-new-Gedeon}
P_n(\mathbf{x}; t) = h_{n}(\mathbf{x})+t \sum_{k=2}^{n} h_{n-k}(\mathbf{x})\, Q_k(\mathbf{x};t).
\end{align}
\end{thm}

The rest of this section is dedicated to the proof of  Theorem~\ref{thm-pq}.
It is known from~\cite{gedeon2016thagomizer} that the polynomial $P_n(\mathbf{x}; t)$ is uniquely determined by the following three conditions:
\begin{itemize}
	\item[(i)] $P_0(\mathbf{x}; t)$=1,
	\item[(ii)] the degree of $P_n(\mathbf{x}; t)$  is less than $(n+1)/2$ for any positive integer $n$,  and 
	\item[(iii)]  for any positive integer $n$ the polynomial $P_n(\mathbf{x}; t)$ satisfies that 
	\begin{align*}
	t^{n+1}P_n(\mathbf{x}; t^{-1}) = &  (t-1) \sum_{\ell =0 }^{n} h_{\ell}[(t-2)X] h_{n-\ell}(\mathbf{x}) \\
	& \ + \sum_{i+j+m=n} P_i(\mathbf{x}; t) h_j[(t-1)X]  h_m[(t-1)X],
	\end{align*}
	where the square bracket  denotes the plethystic substitution~\cite{Haglund2008qt,Haiman2007Geometry} and it is a convention that $X = x_1 + x_2 + \cdots$. 
\end{itemize}
The third condition can also be expressed in terms of  its generating function.
Let  $$\phi(t,u) := \sum_{n=0}^{\infty} P_n(\mathbf{x}; t) u^{n+1}.$$
It is known by Gedeon~\cite[Proposition 4.7]{gedeon2016thagomizer}  that
the condition (iii) is equivalent to say that the function $\phi(t,u)$ satisfies
\begin{align}\label{eq-Gedeon}
\phi(t^{-1},tu) = (t-1)uH(u)v(t,u) + \frac{H(tu)^2}{H(u)^2} \phi(t,u),
\end{align}
	where
\begin{align*}
v(t,u)= \sum_{n=0}^{\infty} h_n[(t-2)X] u^n
\quad  \mbox{(\cite[p. 8]{gedeon2016survey})}.
\end{align*}
We note that the equation~\eqref{eq-Gedeon}  can be simplified as follows.
\begin{lem}\label{lem-phi}
	The function $\phi(t,u)$ satisfies 
	\begin{align}\label{eq-Gedeon-new}
	\phi(t^{-1},tu) = (t-1)u\frac{H(tu)}{H(u)}  +\frac{H(tu)^2}{H(u)^2} \phi(t,u),  
	\end{align}
\end{lem}

\begin{proof}
	It suffices to show that $$v(t,u)= \frac{H(tu)}{H(u)^2}.$$ 
	By the formula~\cite[Theorem 1.27]{Haglund2008qt} 
	\begin{align}\label{hef}
	h_n[E+F]=\sum_{k=0}^{n}{h_k[E]h_{n-k}[F]}, 
	\end{align}
	where $E=E(t_1,t_2,\ldots)$ and $F=F(w_1,w_2,\ldots)$ are two formal series of rational terms in their indeterminates, we have   $$h_n[2X]=\sum_{k=0}^{n}{h_k[X]h_{n-k}[X]}  \mbox{~~and~~}h[tX]=\sum_{k=0}^{n}{h_k[(t-2)X]h_{n-k}[2X]}.$$
	Note that $h_n[X]=h_n(\mathbf{x})$.
	Hence, it follows that 
	$$\sum_{n=0}^{\infty}{h_n[2X]u^n}=\left(\sum_{n=0}^{\infty}{h_n[X]u^n}\right)^2=H(u)^2,$$
	and thus
	$$\sum_{n=0}^{\infty}{h_n[tX]u^n}=\left(\sum_{n=0}^{\infty}{h_n[(t-2)X]u^n}\right) \left(\sum_{n=0}^{\infty}{h_n[2X]u^n}\right)=v(t,u)H(u)^2.$$
	By the definition of plethysm, we have $H(tu)=\sum_{n=0}^{\infty}{h_n[tX]u^n}$. Thus $H(tu)=v(t,u)H(u)^2$ as desired.
	%
	This completes the proof.
\end{proof}

In order to prove  Theorem \ref{thm-pq},  we shall prove that for every positive integer $n$ the polynomial on the right hand side of~\eqref{eq-new-Gedeon} also 
 satisfies the three conditions (i), (ii), and (iii$^{\prime}$).
For convenience, we define  $R_n(\mathbf{x};t)$ as 
\begin{align}\label{def-R}
R_n(\mathbf{x};t) := 
\begin{cases}
1,& n=0,\\
h_{n}(\mathbf{x})+t \sum_{k=2}^{n} h_{n-k}(\mathbf{x})\, Q_k(\mathbf{x};t), & n\ge 1.
\end{cases}
\end{align}
By \eqref{def-R}, we know  $R_0(\mathbf{x}; t)=1$ and  the degree of $R_n(\mathbf{x}; t)$ is $\lfloor \frac{n}{2} \rfloor$. Hence $R_n(\mathbf{x}; t)$ satisfies the first two  conditions.
For the condition (iii$^{\prime}$), 
let us  consider the generating function of $R_n(\mathbf{x}; t)$. Denote $$\rho(t,u):= \sum_{n=0}^{\infty} R_n(\mathbf{x}; t)u^{n+1}$$ and  $$H(u):= \sum_{n=0}^{\infty} h_{n}(\mathbf{x}) u^n.$$
We have the following result.
\begin{lem} \label{lem-rho}
The function $\rho(t,u)$ satisfies 
$$\rho(t^{-1},tu)  =(t-1)  u\frac{H(tu)}{H(u)}+  \frac{H(tu)^2}{H(u)^2} \rho(t,u).$$
\end{lem}

\begin{proof}
Let $\psi(t,u)=\sum_{n=2}^{\infty}Q_n(\mathbf{x}; t) u^{n-1}.$
Since  $Q_{0}(\mathbf{x};t)=Q_{1}(\mathbf{x}; t)=0$, it follows from \eqref{def-R} that 
\begin{align}
\rho(t,u) & = u H(u) +tu^{2}H(u) \psi(t, u) \notag \\[6pt]
& = u H(u) \left( 1+tu \, \psi(t, u) \right) .  \label{eq-rho-psi}
\end{align}
Hence $\rho(t^{-1},tu)$ turns outs to be
\begin{align}\label{rho_tu}
\rho(t^{-1},tu) & = tu H(tu) \left( 1+u \, \psi(t^{-1} ,tu) \right).
\end{align}
On the other hand, taking the coefficient of $x$ in \cite[Equation (4)]{gedeon2017equivariant}, the function $\psi(t, u)$ satisfies
\begin{align*}
\left(\frac{1}{u}+ \psi(t^{-1},tu)  \right) H(u) -\left(\frac{1}{u}+h_{1}(\mathbf{x}) \right)
=\left(\frac{1}{tu}+ \psi(t,u) \right) H(tu) - \left(\frac{1}{tu}+h_{1}(\mathbf{x}) \right).
\end{align*}
Hence, we have that the function $\psi(t,u)$ satisfies the following equation
\begin{align*} 
\psi(t^{-1},tu) &  =  -\frac{1}{u} +\frac{t-1}{tu H(u)}  + \frac{H(tu)}{H(u)}  \left(  \psi(t,u) + \frac{1}{tu} \right),
\end{align*}
and  thus it follows from~\eqref{eq-rho-psi} that 
\begin{align}\label{eq-psi}
\psi(t^{-1},tu) &   =   -\frac{1}{u} +\frac{t-1}{tu H(u)} + \frac{H(tu)}{tu^2H(u)}  \rho(t,u).
\end{align}
Substituting \eqref{eq-psi} into the right hand side of~\eqref{rho_tu}, we have that 
the function $\rho(t,u)$ satisfies that
\begin{align*}
\rho(t^{-1},tu) & =  tu H(tu)+tu^2 H(tu) \left(  -\frac{1}{u} +\frac{t-1}{tu H(u)} + \frac{H(tu)}{tu^2H(u)}  \rho(t,u)  \right) \\[6pt]
& =  (t-1)u \frac{H(tu)}{H(u)}+  \frac{H(tu)^2}{H(u)^2} \rho(t,u).
\end{align*}
This completes the proof.
\end{proof}

We are in the position to prove Theorem~\ref{thm-pq}.
\begin{proof}[Proof of Theorem~\ref{thm-pq}]
As shown previously, the polynomial $R_n(\mathbf{x}; t)$ satisfies the first two  conditions (i) and (ii).
By Lemma~\ref{lem-phi}, The condition (iii) is equivalent to the generating function $\phi(t,u)$ of $P_n(\mathbf{x}; t)$ satisfies \eqref{eq-Gedeon-new}. Compared with  Lemma~\ref{lem-rho},   the generating function $\psi(t,u)$ of $R_n(\mathbf{x}; t)$ satisfies the same function.
Thus, we obtain that the  condition (iii) is true for $R_n(\mathbf{x}; t)$ as well. 
Since these three conditions uniquely determines a polynomial sequence, we get that 
$P_n(\mathbf{x}; t)= R_n(\mathbf{x}; t)$ for every positive integer $n$. 
This completes the proof of Theorem~\ref{thm-pq}.
\end{proof}

\section{Proof of Conjecture \ref{conj-main}}
In this section, we shall prove the following theorem which is equivalent to Conjecture~\ref{conj-main} in the sense of  Frobenius  map.
Our proof is based on the Pieri rule. 
\begin{thm}\label{thm-sym}
For any positive integer $n$, we have 
\begin{align}\label{eq-conj-sym}
P_n(\mathbf{x}; t) = \sum_{\lambda \in \Upsilon_n} \kappa(\lambda) s_{\lambda}(\mathbf{x}) t^{\ell(\lambda)-1} (t+1)^{\omega(\lambda)} + h_{n} \left( \mathbf{x}) ((n-1) t+1 \right).
\end{align}
\end{thm}
\begin{proof}
By Theorem~\ref{thm-pq} we need to prove that $R_n(\mathbf{x}; t)$ is equal to the right side of \eqref{eq-conj-sym}, namely
\begin{align*}
\sum_{\lambda \in \Upsilon_n} \kappa(\lambda) s_{\lambda}(\mathbf{x}) t^{\ell(\lambda)-1} (t+1)^{\omega(\lambda)} + h_{n}(\mathbf{x}) ((n-1) t+1)=h_{n}(\mathbf{x})+t \sum_{k=2}^{n} h_{n-k}(\mathbf{x})\, Q_k(\mathbf{x};t).
\end{align*}
Recall that  
$Q_n(\mathbf{x}; t)=\sum_{i=0}^{\lfloor \frac{n}{2} \rfloor -1} s_{n-2i,2^i}(\mathbf{x}) t^{i}$ for $n\ge 2$. It suffices to prove that
\begin{align}
\sum_{\lambda \in \Upsilon_n} \kappa(\lambda) s_{\lambda}(\mathbf{x}) t^{\ell(\lambda)-1} (t+1)^{\omega(\lambda)} + h_{n}(\mathbf{x}) (n-1) t
= \sum_{k=2}^{n} h_{n-k}(\mathbf{x}) \sum_{i=0}^{\lfloor \frac{k}{2} \rfloor -1} s_{k-2i,2^i}(\mathbf{x}) t^{i+1}.\label{eqab}
\end{align}
For convenience, we denote by $A_n(\mathbf{x}; t)$ and $B_n(\mathbf{x}; t)$ the left side and the right side of~\eqref{eqab}, respectively.

We first show that  $B_n(\mathbf{x}; t)$  is of the form
\begin{align*}
\sum_{\lambda \in \Upsilon_n} s_{\lambda}(\mathbf{x}) a_{\lambda}(t)   + h_{n}(\mathbf{x})  a_{n}(t),
\end{align*}
where $a_{\lambda}(t)$ and $a_{n}(t)$ are polynomials of $t$ with nonnegative integer coefficients.
In fact, by the Pieri rule we have that for $2 \le k \le n$
\begin{align*}
 h_{n-k}(\mathbf{x}) h_{k}(\mathbf{x}) = \sum_{p=\max(0, n-2k)}^{n-k} s_{k+p,n-k-p}(\mathbf{x}) , 
\end{align*}
and for $2 \le k \le n$ and $1\le i \le \lfloor \frac{k}{2} \rfloor-1$
\allowdisplaybreaks
\begin{align*}
 h_{n-k}(\mathbf{x}) s_{k-2i,2^i}(\mathbf{x}) =  & 
  \sum_{p=\max(0, n-2k+2i+2)}^{n-k} s_{k+p-2i,n-k-p+2,2^{i-1}}(\mathbf{x})  \\[6pt]
  & \ + \sum_{p=\max(0,  n-2k+2i+1)}^{n-k-1} s_{k+p-2i,n-k-p+1,2^{i-1},1}(\mathbf{x})\\[6pt]
  & \ +\sum_{p=\max(0, n-2k+2i)}^{n-k-2} s_{k+p-2i,n-k-p,2^{i}}(\mathbf{x}).
\end{align*}
Set
\allowdisplaybreaks
\begin{align*}
B_n^{\langle 1 \rangle}(\mathbf{x}; t) & := \sum_{k=2}^{n} \sum_{p=\max(0,  n-2k)}^{n-k} s_{k+p,n-k-p}(\mathbf{x}) \, t, \\[5pt]
B_n^{\langle 2\rangle}(\mathbf{x}; t) & := \sum_{k=2}^{n} \sum_{i=1}^{\lfloor \frac{k}{2} \rfloor -1}   \sum_{p=\max(0, n-2k+2i+2)}^{n-k} s_{k+p-2i,n-k-p+2,2^{i-1}}(\mathbf{x}) \, t^{i+1}, \\[5pt]
B_n^{\langle 3 \rangle}(\mathbf{x}; t) & := \sum_{k=2}^{n-1} \sum_{i=1}^{\lfloor \frac{k}{2} \rfloor -1} \sum_{p=\max(0, n-2k+2i+1)}^{n-k-1} s_{k+p-2i,n-k-p+1,2^{i-1},1}(\mathbf{x}) \,  t^{i+1}, \\[5pt]
B_n^{\langle 4 \rangle}(\mathbf{x}; t) & := \sum_{k=2}^{n-2} \sum_{i=1}^{\floor{\frac{k}{2}} -1} \sum_{p=\max(0, n-2k+2i)}^{n-k-2} s_{k+p-2i,n-k-p,2^{i}}(\mathbf{x})   \, t^{i+1}.
\end{align*}
Hence,
\begin{align*}
B_n(\mathbf{x}; t)  =B_n^{\langle 1 \rangle}(\mathbf{x}; t)+ B_n^{\langle 2\rangle}(\mathbf{x}; t) + B_n^{\langle 3 \rangle}(\mathbf{x}; t) + B_n^{\langle 4 \rangle}(\mathbf{x}; t).
\end{align*}

We proceed to  prove that $a_{\lambda}(t)$ and $a_{n}(t)$  agree with the corresponding polynomials of $t$ appearing in $A_n(\mathbf{x}; t)$.
Since $h_{n}(\mathbf{x})$ can be obtained only from $B_n^{\langle 1 \rangle}(\mathbf{x}; t)$, where $k$ ranges from $2$ to $n$, we obtain that $a_{n}(t)=(n-1)t.$
We shall prove
$a_{\lambda}(t)=\kappa(\lambda) t^{\ell(\lambda)-1} (t+1)^{\omega(\lambda)}$
for $\lambda \in \Upsilon_n$. To this end, we divide the proof  into the  following three cases according to the definitions of $\kappa(\lambda)$ and $\omega(\lambda)$:

\noindent \textbf{Case 1: }   $\lambda=(n-1,1)$. 
In this case, $s_{n-1,1}(\mathbf{x})$ can be obtained only from  $B_n^{\langle 1 \rangle}(\mathbf{x}; t)$, where $k$ ranges from $2$ to $n-1$. Thus we have $a_{n-1,1}(t)=(n-2)t=\kappa(\lambda) t^{\ell(\lambda)-1} (t+1)^{\omega(\lambda)}$.

\noindent \textbf{Case 2: }  $\lambda_{\ell(\lambda)}=1$ and $\lambda\neq(n-1,1)$. 
In this case, $\lambda$ must be  of the  form $(\lambda_1, \lambda_2,2^{i-1},1)$, where $i=\ell(\lambda)-2 \ge 1$.   Hence, we get that $s_{\lambda}(\mathbf{x})$ can be obtained only from $B_n^{\langle 3 \rangle}(\mathbf{x}; t)$.
We next compute the coefficient of $s_{\lambda_1, \lambda_2,2^{i-1},1} (\mathbf{x})$ in   $B_n^{\langle 3 \rangle}(\mathbf{x}; t)$. 
From the betweenness condition of the Pirie rule, we know that 
$$\lambda_2\le k-2i \le \lambda_1,$$
and thus
$$2< \lambda_2+2i \le k\le \lambda_1+2i=n-\lambda_2+1\leq n.$$ When $i$ and $k$ is fixed, $p$ is uniquely determined, since $p=\lambda_1+2i-k$.
Since  $\lambda_1\geq \lambda_2\ge 2$ and $\lambda_1+\lambda_2=n-2i+1$, we have
$$\ceil{\frac{n+1}{2}}-i\leq \lambda_1\leq n-2i-1,$$
and thus
$$\ceil{\frac{n+1}{2}}-k+i\leq p=\lambda_1+2i-k\leq n-k-1.$$
Hence when $\lambda$ is fixed, 
 $k$ is bounded by  the following inequality $$\lambda_2+2i \le k\le \lambda_1+2i,$$
 and any possible integer $k$ in this interval makes an occurrence of $s_{\lambda}(\mathbf{x})$.
Since $\omega(\lambda)=0$ and $i=\ell(\lambda)-2$, we have  that 
$$a_{\lambda}(t)=(\lambda_1-\lambda_2+1)t^{i+1}=\kappa(\lambda) t^{\ell(\lambda)-1} (t+1)^{\omega(\lambda)}.$$

\noindent \textbf{Case 3: } $\lambda_{\ell(\lambda)}\neq1$.
In this case, we have that 
$\lambda$ is of the  form  $(\lambda_1, \lambda_2, 2^{j})$, where $j=\ell(\lambda)-2\ge 0$.

When $j=0$, $s_{\lambda}(\mathbf{x})$ can be  obtained only  from $B_n^{\langle 2 \rangle}(\mathbf{x}; t)$ and $B_n^{\langle 1 \rangle}(\mathbf{x}; t)$.
When $j\geq 1$, $s_{\lambda}(\mathbf{x})$ can be  obtained only  from $B_n^{\langle 2 \rangle}(\mathbf{x}; t)$ and $B_n^{\langle 4 \rangle}(\mathbf{x}; t)$.  
Note that when $s_{\lambda}(\mathbf{x})$ is  obtained  from  $B_n^{\langle 2 \rangle}(\mathbf{x}; t)$,  $i$ should be $j+1$ and thus $t^{i+1}$ will be $t^{j+2}$. Along similar lines with \textbf{Case 2},
we have that 
\begin{align*}
a_{\lambda}(t)
&=(\lambda_1-\lambda_2+1)\, t^{j+2}+(\lambda_1-\lambda_2+1)\, t^{j+1}\\[4pt]
&=(\lambda_1-\lambda_2+1)\, t^{\ell(\lambda)}+(\lambda_1-\lambda_2+1)\, t^{\ell(\lambda)-1}\\[4pt]
&=(\lambda_1-\lambda_2+1)\, t^{\ell(\lambda)-1}(t+1)\\[4pt]
&=\kappa(\lambda) \, t^{\ell(\lambda)-1} (t+1)^{\omega(\lambda)}.
\end{align*}
%
%

Therefore,  we have shown  that,  for each  partition $\lambda$ of $n$,   the coefficients of $s_{\lambda}$ in $A_n(\mathbf{x}; t)$ and $B_n(\mathbf{x}; t)$  are equal.  Thus  $A_n(\mathbf{x}; t) = B_n(\mathbf{x}; t)$, which completes the proof.
\end{proof}

\noindent{\bf Acknowledgements.} 
The authors would like to thank James Haglund for his helpful discussion about the plethysm of symmetric functions.
This work was supported by the National Science Foundation of China (No.  11701424, 11801447).


\end{document}